\pdfoutput=1
\newif\ifpersonal
\documentclass[10pt,a4paper]{amsart} %reqno places equation numbers on the right
\usepackage{amsmath,amsthm,amssymb,mathrsfs,mathtools,tensor} % math related
\usepackage[all,cmtip]{xy}
\usepackage[LGR,T1]{fontenc} % LGR allows for Greek text input
\usepackage[utf8]{inputenc} 
\usepackage{CJKutf8}

\usepackage{microtype,inconsolata} % latex technical issues
\usepackage{enumerate,comment,braket,xspace,tikz,tikz-cd,csquotes} % utilities
\usetikzlibrary{shapes.geometric}
\usepackage[centering,vscale=0.7,hscale=0.7]{geometry}
\usepackage[hidelinks]{hyperref}
\usepackage[capitalize]{cleveref}
\usepackage{tikzarrows}

\tikzcdset{
	cells={font=\everymath\expandafter{\the\everymath\displaystyle}},
}

\numberwithin{equation}{section}
\theoremstyle{plain}
\newtheorem{thm}[equation]{Theorem}
\newtheorem{lem}[equation]{Lemma}
\newtheorem{prop}[equation]{Proposition}

\newtheorem{cor}[equation]{Corollary}

\theoremstyle{definition}
\newtheorem{defin}[equation]{Definition}
\newtheorem{notation}[equation]{Notation}
\newtheorem{eg}[equation]{Example}
\newtheorem{rem}[equation]{Remark}

\newtheorem{recollection}[equation]{Recollection}

\DeclareMathOperator{\Mon}{Mon}
\renewcommand{\Bar}{\operatorname{Bar}}
\DeclareMathOperator{\CoBar}{CoBar}
\DeclareMathOperator{\gp}{gp}
\ifpersonal
\newcommand{\personal}[1]{\textcolor[rgb]{0,0,1}{(Personal: #1)}}
\newcommand{\discussion}[1]{\textcolor{violet}{(Discussion: #1)}}
\else
\newcommand{\personal}[1]{\ignorespaces}
\newcommand{\discussion}[1]{\ignorespaces}
\fi

\newcommand{\Z}{\mathbb Z}

\newcommand{\cC}{\mathcal C}
\newcommand{\cD}{\mathcal D}
\newcommand{\cE}{\mathcal E}

\newcommand{\cX}{\mathcal X}
\newcommand{\cY}{\mathcal Y}

\DeclareFontFamily{U}{BOONDOX-calo}{\skewchar\font=45 }
\DeclareFontShape{U}{BOONDOX-calo}{m}{n}{<-> s*[1.05] BOONDOX-r-calo}{}
\DeclareFontShape{U}{BOONDOX-calo}{b}{n}{<-> s*[1.05] BOONDOX-b-calo}{}
\DeclareMathAlphabet{\mathcalboondox}{U}{BOONDOX-calo}{m}{n}
\makeatletter
\let\save@mathaccent\mathaccent
\newcommand*\if@single[3]{%
	\setbox0\hbox{${\mathaccent"0362{#1}}^H$}%
	\setbox2\hbox{${\mathaccent"0362{\kern0pt#1}}^H$}%
	\ifdim\ht0=\ht2 #3\else #2\fi
}
\newcommand*\rel@kern[1]{\kern#1\dimexpr\macc@kerna}
\newcommand*\widebar[1]{\@ifnextchar^{{\wide@bar{#1}{0}}}{\wide@bar{#1}{1}}}
\newcommand*\wide@bar[2]{\if@single{#1}{\wide@bar@{#1}{#2}{1}}{\wide@bar@{#1}{#2}{2}}}
\newcommand*\wide@bar@[3]{%
	\begingroup
	\def\mathaccent##1##2{%
		\let\mathaccent\save@mathaccent
		\if#32 \let\macc@nucleus\first@char \fi
		\setbox\z@\hbox{$\macc@style{\macc@nucleus}_{}$}%
		\setbox\tw@\hbox{$\macc@style{\macc@nucleus}{}_{}$}%
		\dimen@\wd\tw@
		\advance\dimen@-\wd\z@
		\divide\dimen@ 3
		\@tempdima\wd\tw@
		\advance\@tempdima-\scriptspace
		\divide\@tempdima 10
		\advance\dimen@-\@tempdima
		\ifdim\dimen@>\z@ \dimen@0pt\fi
		\rel@kern{0.6}\kern-\dimen@
		\if#31
		\overline{\rel@kern{-0.6}\kern\dimen@\macc@nucleus\rel@kern{0.4}\kern\dimen@}%
		\advance\dimen@0.4\dimexpr\macc@kerna
		\let\final@kern#2%
		\ifdim\dimen@<\z@ \let\final@kern1\fi
		\if\final@kern1 \kern-\dimen@\fi
		\else
		\overline{\rel@kern{-0.6}\kern\dimen@#1}%
		\fi
	}%
	\macc@depth\@ne
	\let\math@bgroup\@empty \let\math@egroup\macc@set@skewchar
	\mathsurround\z@ \frozen@everymath{\mathgroup\macc@group\relax}%
	\macc@set@skewchar\relax
	\let\mathaccentV\macc@nested@a
	\if#31
	\macc@nested@a\relax111{#1}%
	\else
	\def\gobble@till@marker##1\endmarker{}%
	\futurelet\first@char\gobble@till@marker#1\endmarker
	\ifcat\noexpand\first@char A\else
	\def\first@char{}%
	\fi
	\macc@nested@a\relax111{\first@char}%
	\fi
	\endgroup
}
\makeatother

\newcommand{\PSh}{\mathrm{PSh}}
\newcommand{\Sh}{\mathrm{Sh}}
\newcommand{\HSh}{\mathrm{Sh}^{\mathrm{hyp}}}

\newcommand{\Mod}{\mathrm{Mod}}

\newcommand{\cS}{\categ{Spc}}

\newcommand{\categ}[1]{\textbf{\textup{#1}}}
\newcommand{\fib}{\mathrm{fib}}

\newcommand{\PrL}{\categ{Pr}^{\mathrm{L}}}

\newcommand{\PrR}{\mathcal P \mathrm{r}^{\mathrm{R}}}

\newcommand{\ev}{\mathrm{ev}}

\newcommand{\id}{\mathrm{id}}

\newcommand{\op}{^\mathrm{op}}

\usetikzlibrary{decorations.markings} %arrows for open immersions and closed immersions
\tikzset{
  closed/.style = {decoration = {markings, mark = at position 0.5 with { \node[transform shape, xscale = .8, yscale=.4] {/}; } }, postaction = {decorate} },
  open/.style = {decoration = {markings, mark = at position 0.5 with { \node[transform shape, scale = .7] {$\circ$}; } }, postaction = {decorate} }
}

\DeclareMathOperator{\Fun}{Fun}
\DeclareMathOperator{\FunR}{Fun^R}

\DeclareMathOperator{\Map}{Map}

\DeclareMathOperator{\Sp}{Sp}

\DeclareMathOperator*{\colim}{colim}

\crefname{prop}{Proposition}{Propositions}
\Crefname{prop}{Proposition}{Propositions}

\renewcommand{\Sp}{\mathbf{Sp}}
\DeclareMathOperator{\Stab}{Sp}
\newcommand{\EE}{\mathbb{E}}
\newcommand{\ZZ}{\mathbb{Z}}

\renewcommand{\PrR}{\mathbf{Pr}^{\mathrm{R}}}
\renewcommand{\PrL}{\mathbf{Pr}^{\mathrm{L}}}

\newcommand{\period}{\rlap{\  .}}
\newcommand{\comma}{\rlap{\  ,}}

\begin{document}
\title{Standard $t$-structures}

\author{Peter J. Haine}
\address{Peter J. Haine, Department of Mathematics, University of California, Evans Hall, Berkeley, CA 94720, USA}
\email{peterhaine@berkeley.edu}

\author{Mauro Porta}
\address{Mauro Porta, Institut de Recherche Mathématique Avancée, 7 Rue René Descartes, 67000 Strasbourg, France}
\email{porta@math.unistra.fr}

\author{Jean-Baptiste Teyssier}
\address{Jean-Baptiste Teyssier, Institut de Mathématiques de Jussieu, 4 place Jussieu, 75005 Paris, France}
\email{jean-baptiste.teyssier@imj-prg.fr}
%\date{January 23, 2018}


\begin{abstract}
	We provide a general construction of induced $t$-structures, that generalizes standard $t$-structures for $\infty$-categories of sheaves.
	More precisely, given a presentable $\infty$-category $\cX$ and a presentable stable $\infty$-category $\cE$ equipped with an accessible $t$-structure $\tau = (\cE_{\geqslant 0}, \cE_{\leqslant 0})$, we show that $\cX \otimes \cE$ is equipped with a canonical $t$-structure whose coconnective part is given in $\cX \otimes \cE_{\leqslant 0}$.
	When $\cX$ is an $\infty$-topos, we give a more explicit description of the connective part as well.
\end{abstract}

\maketitle

%\personal{PERSONAL COMMENTS ARE SHOWN!!!}

\tableofcontents

%-------------------------------------------------------------------%
%-------------------------------------------------------------------%
%  Introduction                                                     %
%-------------------------------------------------------------------%
%-------------------------------------------------------------------%

\section{Introduction}

Let $X$ be a topological space and let $ A $ be a connective $\EE_1$-ring spectrum.
Then the $ \infty $-categories $\Sh(X;\Mod_A)$ and $\HSh(X;\Mod_A)$ of sheaves and hypersheaves of $ A $-modules on $ X $ inherit an induced $t$-structure, called the \textit{standard $t$-structure}.
See \cite[\S1.3.2]{Lurie_SAG}.
In this short note, using properties of the tensor product of presentable $\infty$-categories, we revisit this construction in greater generality.
Given a presentable stable $\infty$-category $\cE$ equipped with a $t$-structure $ (\cE_{\geqslant 0}, \cE_{\leqslant 0})$, we show that under very mild conditions, the $\infty$-categories $\Sh(X;\cE)$ and $\HSh(X;\cE)$ of $ \cE $-valued sheaves and hypersheaves inherit an induced $t$-structure, that generalizes the case $\cE = \Mod_A$.
Along the way, we establish a certain exactness property of this operation that seems interesting on its own right, and that we briefly describe now.

\medskip

Given a presentable stable $\infty$-category $\cE$ equipped with an accessible $t$-structure $ (\cE_{\geqslant 0}, \cE_{\leqslant 0})$, both $\cE_{\geqslant 0}$ and $\cE_{\leqslant -1}$ are presentable, and the square
\begin{equation*} 
	\begin{tikzcd}
		\cE_{\geqslant 0} \arrow[r, hooked, "i_{\geqslant n}"] \arrow{d} & \cE \arrow{d}{\tau_{\leqslant -1}} \\
		0 \arrow{r} & \cE_{\leqslant n-1}
	\end{tikzcd} 
\end{equation*}
is both a pullback and a pushout square in $\PrL$ (see \cref{recollection:t_structure}).
Tensoring with a second presentable $\infty$-category $\cX$, we obtain a pushout square
\begin{equation}\label{eq:tensoring_t_structure}
	\begin{tikzcd}
		\cX \otimes \cE_{\geqslant 0} \arrow{d} \arrow{r}{\cX \otimes i_{\geqslant 0}} & \cX \otimes \cE \arrow{d}{\cX \otimes \tau_{\leqslant -1}} \\
		0 \arrow{r} & \cX \otimes \cE_{\leqslant-1} \period
	\end{tikzcd}
\end{equation}
However, in general this square is not a pullback.
Nevertheless,  $\cX \otimes \tau_{\leqslant -1}$ is a localization functor (see \cref{lem:tensor_product_localization}) and we can therefore identify $\cX \otimes \cE_{\leqslant -1}$ with a full subcategory of $\cX \otimes \cE$.
We have:

\begin{thm}[{\Cref{cor:standard_t_structure_fundamentals,cor:standard_t_structure_topos}}]\label{intro_thm:main_thm}
	Let $ \cX $ be a presentable $\infty$-category and let $\cE$ be a presentable stable $\infty$-category equipped with an accessible $t$-structure $ (\cE_{\geqslant 0}, \cE_{\leqslant 0})$.
	\begin{enumerate}\itemsep=0.2cm
		\item The presentable stable $\infty$-category $\cX \otimes \cE$ admits a $t$-structure of the form
		\begin{equation*} 
			\big( (\cX \otimes \cE)_{\geqslant 0}, \cX \otimes \cE_{\leqslant 0} \big) \comma 
		\end{equation*}
		where $(\cX \otimes \cE)_{\geqslant 0}$ is the full subcategory of $\cX \otimes \cE$ generated under colimits and extensions by the essential image of the functor $\cX \otimes i_{\geqslant 0} \colon \cX \otimes \cE_{\geqslant 0} \to \cX \otimes \cE$.
		
		\item If $\cX$ is an $\infty$-topos, \eqref{eq:tensoring_t_structure} is a pullback square.
		In particular $\cX \otimes \cE_{\geqslant 0}$ becomes a full subcategory of $\cX \otimes \cE$, and the $t$-structure of the previous point simply becomes
		\begin{equation*} 
			\big( \cX \otimes \cE_{\geqslant 0}, \cX \otimes \cE_{\leqslant 0} \big) \period 
		\end{equation*}
	\end{enumerate}
\end{thm}

\personal{(Mauro) This is hard speculation but I somehow feel it appealing. We can get rid of it, or try to give it some better form to make it into a selling point.}
We conclude the introduction by describing a heuristic speculation.
When $\cX \simeq \PSh(\cC)$ is a presheaf category, one has
\begin{equation*} 
	\PSh(\cC) \otimes \cY \simeq \Fun(\cC\op, \cY) \comma 
\end{equation*}
and in particular it follows that $\PSh(\cC) \otimes (-)$ commutes with both limits and colimits in $\PrL$.
On the other hand, any presentable $\infty$-category $\cX$ can be written as a localization of a presheaf category $\PSh(\cC_0)$.
If Vopěnka's principle holds, the left orthogonal complement of $\cX$ inside $\PSh(\cC)$ is also be presentable, and it is therefore be possible to write it as the localization of a second presheaf category $\PSh(\cC_1)$.
Iterating this process, we construct a ``resolution'' of $\cX$ by presheaf $ \infty $-categories, which have an exact behavior with respect to the tensor product in $\PrL$.
We would then be able to define $\mathsf{Tor}$-categories, and \Cref{intro_thm:main_thm} would be stating that when $\cX$ is an $\infty$-topos,
\begin{equation*} 
	\mathsf{Tor}_1(\cX, \cE_{\leqslant 0}) = 0 \period
\end{equation*}

%-------------------------------------------------------------------%
%  Acknowledgments                                                  %
%-------------------------------------------------------------------%

\subsection*{Acknowledgments}

PH gratefully acknowledges support from the NSF Mathematical Sciences Postdoctoral Research Fellowship under Grant \#DMS-2102957. 

%-------------------------------------------------------------------%
%-------------------------------------------------------------------%
%  Standard t-structures for presentable ∞-categories               %
%-------------------------------------------------------------------%
%-------------------------------------------------------------------%

\section{Standard $t$-structures for presentable $\infty$-categories}\label{subsec:t_structures_presentable}

%-------------------------------------------------------------------%
%  Reminders on the tensor product in Prᴸ                           %
%-------------------------------------------------------------------%

\subsection{Reminders on the tensor product in $\PrL$}

Before we begin, we recall some standard properties of the tensor product in $\PrL$.

\begin{lem}\label{lem:tensor_product_localization}
	Let $L \colon \cY \to \cY'$ be a localization functor in $\PrL$.
	For every $\cX \in \PrL$, the induced functor $\id_\cX \otimes L \colon \cX \otimes \cY \to \cX \otimes \cY'$ is again a localization.
\end{lem}

\begin{proof}
	We have to check that the right adjoint to $\id_\cX \otimes L$ is fully faithful.
	Write $j \colon \cY' \to \cY$ for the fully faithful right adjoint to $L$.
	Using the identifications
	\begin{equation*} 
		\cX \otimes \cY \simeq \FunR(\cX\op, \cY) \qquad \text{and} \qquad \cX \otimes \cY' \simeq \FunR(\cX\op, \cY') 
	\end{equation*}
	provided by \cite[Proposition 4.8.1.17]{Lurie_Higher_algebra}, we see that the right adjoint to $\id_\cX \otimes L$ is identified with the functor induced by composition with $j$.
	In particular, it follows from \cite[Lemma 5.2]{Gepner_Lax_colimits} that it is fully faithful.
\end{proof}

\begin{recollection}\label{recollection:Yoneda_with_coefficients}
	Let $\cX$ be a presentable $\infty$-category and let $X \in \cX$ be an object.
	It follows from \cite[Corollary 4.4.4.9]{HTT} that the functor $ \Map_\cX(X,-) \colon \cX \to \cS $ admits a left adjoint, which we denote by 
	\begin{equation*}
		i_X \coloneqq (-) \otimes X \colon \cS \to \cX \period 
	\end{equation*}
	The functor $ i_X $ is the unique colimit-preserving functor $ \cS \to \cX $ sending the terminal object to $ X $.
	Now let $\cY$ be a second presentable $\infty$-category and consider the induced functor
	\begin{equation*} 
		i_X \otimes \id_\cY \colon \cY \to \cX \otimes \cY \period 
	\end{equation*}
	This functor takes an object $Y \in \cY$ to the elementary tensor $X \otimes Y$. 
	Under the identifications
	\begin{equation*}
		\cX \otimes \cY \simeq \FunR(\cY\op, \cX) \qquad \text{and} \qquad \cY \simeq \FunR(\cY\op, \cS)
	\end{equation*}
	provided by \cite[Proposition 4.8.1.17]{Lurie_Higher_algebra}, we see that the right adjoint to $i_X \otimes \id_\cY$ can be explicitly described as the functor given by postcomposition with $\Map_\cX(X,-)$.
	Unraveling the definitions, we can alternatively describe this right adjoint as the evaluation functor
	\begin{align*}
		\ev_X \colon \cX \otimes \cY \simeq \FunR(\cX\op, \cY) &\to \cY  \\ 
		F &\mapsto F(X) \period
	\end{align*}
	In particular, we deduce that for every $F \in \cX \otimes \cY$, one has a natural identification
	\begin{equation}\label{eq:Yoneda_with_coefficients}
		\Map_{\cX \otimes \cY}( X \otimes Y, F ) \simeq \Map_\cY(Y, F(X)) \comma
	\end{equation}
	where $F$ is viewed as a limit-preserving functor $F \colon \cX\op \to \cY$.
\end{recollection}

%-------------------------------------------------------------------%
%  Standard t-structures                                            %
%-------------------------------------------------------------------%

\subsection{Standard $t$-structures}

\begin{recollection}\label{recollection:t_structure}
	Let $\cE$ be a presentable stable $\infty$-category equipped with an accessible $t$-structure $\tau = (\cE_{\geqslant 0}, \cE_{\leqslant 0})$.
	For every $n \in \ZZ$, the full subcategory $\cE_{\leqslant n}$ is an accessible localization of $\cE$, and therefore it is itself presentable.
	Notice that
	\begin{equation*} 
		\begin{tikzcd}
			\cE_{\geqslant n} \arrow[r, "i_{\geqslant n}", hooked] \arrow{d} & \cE \arrow{d}{\tau_{\leqslant n-1}} \\
			0 \arrow{r} & \cE_{\leqslant n-1}
		\end{tikzcd} 
	\end{equation*}
	is a pullback square.
	In particular, $\cE_{\geqslant n}$ is presentable as well.
	It automatically follows that the above square is also a pushout in $\PrL$.
\end{recollection}

The following lemma follows immediately from the fact that in a stable $\infty$-category a null sequence is a fiber sequence if and only if it is a cofiber sequence:

\begin{lem}\label{lem:kernels_stable_under_extensions}
	Let $f \colon \cE \to \cY$ be a functor between presentable $\infty$-categories.
	Assume that $\cE$ is stable, that $\cY$ is pointed (with zero object $0_\cY$) and that $f$ is either left exact or right exact.
	Then
	\begin{equation*} 
		\ker(f) \coloneqq \{ E \in \cE \mid f(E) \simeq 0_\cY \} 
	\end{equation*}
	is closed under extensions.
	In particular, if $\tau = (\cE_{\geqslant 0}, \cE_{\leqslant 0})$ is a $t$-structure on $\cE$, then for every $n \in \Z$ the full subcategories $\cE_{\leqslant n}$ and $\cE_{\geqslant n}$ are closed under extensions.
\end{lem}

\personal{
\begin{proof}
	Let $E' \to E \to E''$ be a fiber sequence in $\cE$ and assume that $E'$ and $E''$ belong to $\ker(f)$.
	First, assume that $f$ is left exact.
	Applying $f$ to $E' \to E \to E''$, we see that the diagram
	\begin{equation*} 
		\begin{tikzcd}[ampersand replacement=\&]
			f(E') \arrow{r} \arrow{d} \& f(E) \arrow{d} \\
			0_\cY \arrow{r} \& f(E'')
		\end{tikzcd} 
	\end{equation*}
	is a pullback in $\cY$.
	By assumption, the bottom horizontal map is an equivalence, so the same goes for the top horizontal one.
	On the other hand, $f(E') \simeq 0_\cY$ by assumption.
	So, it follows that $E \in \ker(f)$ as well.
	When $f$ is right exact, the same reasoning applies, observing that $E' \to E \to E''$ is at the same time a cofiber sequence because $\cE$ is stable.
\end{proof}}

\begin{notation}\label{ntn:standard_notation}
	Let $\cX$ be a presentable $\infty$-category and let $\cE$ be a presentable stable $\infty$-category equipped with an accessible $t$-structure $\tau = (\cE_{\geqslant 0}, \cE_{\leqslant 0})$.
	For each $n \in \Z$, set
	\begin{equation*} 
		i_{\geqslant n}^\cX \coloneqq \id_\cX \otimes i_{\geqslant n} \colon \cX \otimes \cE_{\geqslant n} \to \cX \otimes \cE \qquad \text{and} \qquad \tau_{\leqslant n}^\cX \coloneqq \id_\cX \otimes \tau_{\leqslant n} \colon \cX \otimes \cE \to \cX \otimes \cE_{\leqslant n} \period 
	\end{equation*}
	Then the square
	\begin{equation*} 
		\begin{tikzcd}
			\cX \otimes \cE_{\geqslant n} \arrow{d} \arrow{r}{i_{\geqslant n}^\cX} & \cX \otimes \cE \arrow{d}{\tau_{\leqslant n-1}^\cX} \\
			0 \arrow{r} & \cX \otimes \cE_{\leqslant n-1}
		\end{tikzcd} 
	\end{equation*}
	is a pushout in $\PrL$; however, it typically is not a pullback.
	It follows from \cref{lem:tensor_product_localization} that the right adjoint to $\tau_{\leqslant n-1}^\cX$ is fully faithful.
	Under the identifications
	\begin{equation*} 
		\cX \otimes \cE \simeq \FunR(\cX\op, \cE) \qquad \text{and} \qquad \cX \otimes \cE_{\leqslant n-1} \simeq \FunR(\cX\op, \cE_{\leqslant n-1}) \comma 
	\end{equation*}
	we see that the right adjoint to $\tau_{\leqslant n-1}^\cX$ is given by composition with $i_{\leqslant n-1}$.
	In particular, $\cX \otimes \cE_{\leqslant n-1}$ is naturally identified with the full subcategory of $\FunR(\cX\op, \cE)$ spanned by those right adjoints $F \colon \cX\op \to \cE$ that factor through $\cE_{\leqslant n-1}$.
	Define
	\begin{equation*} 
		(\cX \otimes \cE)_{\geqslant n} \coloneqq \ker\big( \tau_{\leqslant n-1}^\cX \colon \cX \otimes \cE \to \cX \otimes \cE_{\leqslant n-1} \big) \period 
	\end{equation*}
\end{notation}

\begin{lem}\label{lem:t_structure_standard_properties}
	In the setting of \Cref{ntn:standard_notation}:
	\begin{enumerate}\itemsep=0.2cm
		\item For each $n \in \Z$, both $\cX \otimes \cE_{\leqslant n}$ and $(\cX \otimes \cE)_{\geqslant n}$ are closed under extensions in $ \cX \otimes \cE $.
		
		\item Let $\cX_\bullet \colon I \to \PrL$ be a diagram with limit $\cX$. Assume that for every transition morphism $i \to j$, the induced functor $\cX_i \to \cX_j$ is both a left and a right adjoint.
		Then for each $n \in \Z$, the natural functors
		\begin{equation*} 
			\cX \otimes \cE_{\leqslant n} \to \lim_{i \in I} \cX_i \otimes \cE_{\leqslant n}
		\end{equation*}
		and
		\begin{equation*} 
			\lim_{i \in I} (\cX_i \otimes \cE)_{\geqslant 0} \to \Big( \big( \lim_{i \in I} \cX_i \big) \otimes \cE \Big)_{\geqslant 0} 
		\end{equation*}
		are equivalences.
	\end{enumerate}
\end{lem}

\begin{proof}
	First we prove (1).
	The claim about $(\cX \otimes \cE)_{\geqslant n}$ is a simple consequence of the definitions and \cref{lem:kernels_stable_under_extensions}.
	We now deal with $\cX \otimes \cE_{\leqslant n}$.
	Consider a fiber sequence
	\begin{equation*} 
		\begin{tikzcd}
			F' \arrow{r} \arrow{d} & F \arrow{d} \\
			0 \arrow{r} & F''
		\end{tikzcd}
	\end{equation*}
	in $\cX \otimes \cE$.
	Assume first that both $F'$ and $F''$ belong to $\cX \otimes \cE_{\leqslant n}$.
	Under the identification $\cX \otimes \cE \simeq \FunR(\cX\op, \cE)$, observe that limits are computed objectwise.
	In other words, for every $X \in \cX$, the induced square
	\begin{equation*}
		\begin{tikzcd}
			F'(X) \arrow{r} \arrow{d} & F(X) \arrow{d} \\
			0 \arrow{r} & F''(X)
		\end{tikzcd}
	\end{equation*}
	is a pullback in $\cE$.
	The assumption implies that both $F'(X)$ and $F''(X)$ belong to $\cE_{\leqslant n}$, so the same is true of $F(X)$.
	In other words, $F \in \cX \otimes \cE_{\leqslant n}$.
	
	We now prove (2).
	Since limits commute with limits, it is enough to prove the statement concerning $\cX \otimes \cE_{\leqslant n}$.
	However, the assumption on the diagram and \cite[Proposition 5.5.3.13 \& Theorem 5.5.3.18]{HTT} imply that the limit can equally be computed in $\PrR$.
	The conclusion now follows from \cite[Remark 4.8.1.24]{Lurie_Higher_algebra} (see also \cite[Lemma 4.2.2]{Beyond_conicality}).
\end{proof}

The following is our main result about a $ t $-structure on $ \cX \otimes \cE $ when $ \cX $ is an arbitrary presentable $ \infty $-category.

\begin{prop}\label{cor:standard_t_structure_fundamentals}
	In the setting of \Cref{ntn:standard_notation}, there exists a unique $t$-structure
	\begin{equation*}
		\big( (\cX \otimes \cE)_{\geqslant 0}, (\cX \otimes \cE)_{\leqslant 0} \big)
	\end{equation*}
	on $ \cX \otimes \cE $ whose connective part coincides with the full subcategory $(\cX \otimes \cE)_{\geqslant 0}$ of $\cX \otimes \cE$ introduced in \Cref{ntn:standard_notation}.
	In addition:
	\begin{enumerate}\itemsep=0.2cm
		\item We have $(\cX \otimes \cE)_{\leqslant 0} = \cX \otimes \cE_{\leqslant 0} $ as full subcategories of $ \cX \otimes \cE $.

		\item The connective part $(\cX \otimes \cE)_{\geqslant 0}$ is generated under colimits and extensions by objects of the form $X \otimes E$ for $X \in \cX$ and $E \in \cE_{\geqslant 0}$.
	\end{enumerate}
\end{prop}

\begin{proof}
	It follows from \cref{lem:t_structure_standard_properties}-(1) that $(\cX \otimes \cE)_{\geqslant 0}$ is a full subcategory of $\cX \otimes \cE$ closed under colimits and extensions. 
	In particular, \cite[Proposition 1.4.4.11-(1)]{Lurie_Higher_algebra} applies, providing the existence (and uniqueness) of the required $t$-structure.
	The definition of $(\cX \otimes \cE)_{\leqslant 0}$ shows that
	\begin{equation*} 
		\cX \otimes \cE_{\leqslant -1} \subseteq (\cX \otimes \cE)_{\leqslant -1} \period 
	\end{equation*}
	To prove that equality holds, let $F \in (\cX \otimes \cE)_{\leqslant -1}$, and view $ F $ as a right adjoint functor $\cX\op \to \cE$.
	We have to prove that $F$ factors through $\cE_{\leqslant -1}$.
	Observe that the functoriality of the tensor product of $\infty$-categories implies that the composite
	\begin{equation*} 
		\begin{tikzcd}[sep=4em]
			\cX \otimes \cE_{\geqslant 0} \arrow[r, "\id_{\cX} \otimes i_{\geqslant 0}"] & \cX \otimes \cE \arrow[r, "\id_{\cX} \otimes \tau_{\leqslant -1}"] & \cX \otimes \cE_{\leqslant -1} 
		\end{tikzcd}
	\end{equation*}
	is zero.
	In other words, the functor $\id_\cX \otimes i_{\geqslant 0} \colon \cX \otimes \cE_{\geqslant 0} \to \cX \otimes \cE$ factors through $(\cX \otimes \cE)_{\geqslant 0}$.
	It follows that every object of the form $X \otimes E$, for $X \in \cX$ and $E \in \cE_{\geqslant 0}$, belongs to $(\cX \otimes \cE)_{\geqslant 0}$.
	In particular, the assumption $F \in (\cX \otimes \cE)_{\leqslant -1}$ guarantees that
	\begin{equation*} 
		\Map_\cE( E, F(X) ) \simeq \Map_{\cX \otimes \cE}( X \otimes E, F ) \simeq 0 \comma 
	\end{equation*}
	where the first equivalence follows from \cref{recollection:Yoneda_with_coefficients}, specifically \eqref{eq:Yoneda_with_coefficients}.
	Since this holds for every $X \in \cX$ and every $E \in \cE_{\geqslant 0}$, we deduce that $F$ factors through $\cE_{\leqslant -1}$.
	Thus, $\cX \otimes \cE_{\leqslant -1} = (\cX \otimes \cE)_{\leqslant -1}$.
	
	We now prove item (2). 
	Let $\cC \subset \cX \otimes \cE $ be the smallest full subcategory closed under colimits and extensions and containing objects of the form $X \otimes E$ for $X \in \cX$ and $E \in \cE_{\geqslant 0}$.
	Recall from \cite[Proposition 1.4.4.11-(2)]{Lurie_Higher_algebra} that $\cC$ is automatically presentable, and that therefore it gives rise to a $t$-structure $\tau' = (\cC, \cD)$ on $\cX \otimes \cE$.
	The same argument given above immediately implies that $\cD \subseteq \cX \otimes \cE_{\leqslant 0}$.
	Conversely, let $F \in \cX \otimes \cE_{\leqslant 0}$ and let $\cC_F$ be the full subcategory of $\cX \otimes \cE$ spanned by the objects $G$ such that $\Map_{\cX \otimes \cE}(G,F) \simeq 0$.
	By definition, $\cC_F$ is closed under colimits, and \cref{lem:kernels_stable_under_extensions} implies that $ \cC_F $ is closed under extensions as well.
	Moreover, $ \cC_F $ contains every object of the form $X \otimes E$ for $X \in \cX$ and $E \in \cE_{\geqslant 1}$.
	Thus, $ \cC_F $ contains $\cC[-1]$.
	It follows that $F \in \cD$, and hence that $\cD = \cX \otimes \cE_{\leqslant 0}$.
	The uniqueness of the $t$-structure implies then that $\cC = (\cX \otimes \cE)_{\geqslant 0}$, whence the conclusion.
\end{proof}

\begin{defin}
	In the setting of \Cref{ntn:standard_notation}, we refer to
	\begin{equation*}
		\tau^{\cX} \coloneqq \big( (\cX \otimes \cE)_{\geqslant 0}, (\cX \otimes \cE)_{\leqslant 0} \big)
	\end{equation*}
	as the \emph{standard $t$-structure} on $\cX \otimes \cE$ induced by the $t$-structure $\tau = (\cE_{\geqslant 0}, \cE_{\leqslant 0})$ on $\cE$.
	\personal{(Mauro) The notation $\tau^{\cX}$ is a bit awkward -- $\tau_\cX$ would perhaps look nicer. But it's compatible with the notation for the truncation functor already used $\tau^\cX_{\leqslant n}$ and $\tau^\cX_{\geqslant n}$, so I'd rather keep the upper $\cX$.}
\end{defin}

\begin{eg}\label{eg:sheaves_of_spectra}
	Let $\cX$ be an $\infty$-topos and let $\cE = \Sp$ be the $\infty$-category of spectra, equipped with its standard $t$-structure.
	In \cite[Proposition 1.3.2.7]{Lurie_SAG} it is shown that $\cX \otimes \Sp \simeq \Sh(\cX;\Sp)$ is equipped with a $t$-structure $(\Sh(\cX;\Sp)_{\geqslant 0}, \Sh(\cX;\Sp)_{\leqslant 0})$.
	In addition, \cite[Remark 1.3.2.6]{Lurie_SAG} provides a natural identification
	\begin{equation*} 
		\Sh(\cX;\Sp)_{\leqslant 0} \simeq \Sh(\cX;\Sp_{\leqslant 0}) \simeq \cX \otimes \Sp_{\leqslant 0} \period
	\end{equation*}
	It follows that the $t$-structure on $\cX \otimes \Sp$ coincides with the one provided by \cref{cor:standard_t_structure_fundamentals}.
	In particular, it follows from \cite[Proposition 1.3.2.7]{Lurie_SAG} that this $t$-structure is compatible with filtered colimits and right complete.
\end{eg}

\begin{cor}\label{cor:t_structure_right_exact}
	Let $f^{*} \colon \cX \to \cY$ be a functor in $\PrL$.
	Then the induced functor
	\begin{equation*} 
		f^{*}_\cE \coloneqq f^{*} \otimes \id_\cE \colon \cX \otimes \cE \to \cY \otimes \cE 
	\end{equation*}
	is right $t$-exact.
	If in addition $f^{*}$ is a left exact left adjoint between $\infty$-topoi and $\cE = \Sp$, then $f^{*}_{\cE} $ is also left $t$-exact.
\end{cor}

\begin{proof}
	Write $f_{*} \colon \cY \to \cX$ for the right adjoint to $ f^{*} $.
	Under the identifications
	\begin{equation*} 
		\cX \otimes \cE \simeq \FunR(\cX\op, \cE) \qquad \text{and} \qquad \cY \otimes \cE \simeq \FunR(\cY\op, \cE) \comma 
	\end{equation*}
	we see that the right adjoint $f_{*}^{\cE}$ to $ f^{*}_\cE$ is given by composition with $f_{*}$.
	It immediately follows that $f_{*}^{\cE}$ takes $\cY \otimes \cE_{\leqslant 0}$ to $\cX \otimes \cE_{\leqslant 0}$, and therefore that $f^{*}_\cE$ is right $t$-exact.
	The second half of the statement follows combining \cref{eg:sheaves_of_spectra} with \cite[Remark 1.3.2.8]{Lurie_SAG}.
\end{proof}

%-------------------------------------------------------------------%
%-------------------------------------------------------------------%
%  Standard t-structures for ∞-topoi                                %
%-------------------------------------------------------------------%
%-------------------------------------------------------------------%

\section{Standard $t$-structures for $\infty$-topoi}

Let $\cE$ be a presentable stable $\infty$-category equipped with an accessible $t$-structure $\tau = (\cE_{\geqslant 0}, \cE_{\leqslant 0})$.
We now investigate the natural comparison functor
\begin{equation*}\label{eq:t_structure_comparison_map}
	i_{\geqslant 0}^{\cX} \colon \cX \otimes \cE_{\geqslant 0} \to (\cX \otimes \cE)_{\geqslant 0} \period
\end{equation*}
The main result of this section is that when $\cX$ is an $\infty$-topos and the $t$-structure $\tau$ is right complete, this functor is an equivalence (\Cref{cor:standard_t_structure_topos}).

We begin with the following general criterion, that holds without extra assumptions:

\begin{lem}\label{lem:tensor_product_connectives_extensions}
	Let $\cE$ be a presentable stable $\infty$-category equipped with an accessible $t$-structure $\tau = (\cE_{\geqslant 0}, \cE_{\leqslant 0})$ and let $ \cX $ be a presentable $ \infty $-category.
	Then the following conditions are equivalent:
	\begin{enumerate}\itemsep=0.2cm
		\item The functor $ i_{\geqslant 0}^{\cX} \colon \cX \otimes \cE_{\geqslant 0} \to \cX \otimes \cE $ is fully faithful and the essential image of $ i_{\geqslant 0}^{\cX} $ is closed under extensions.

		\item The functor $ i_{\geqslant 0}^{\cX} \colon \cX \otimes \cE_{\geqslant 0} \to (\cX \otimes \cE)_{\geqslant 0} $ is an equivalence.
		
		\item There exists an integer $ n \in \ZZ $ such that the $ \infty $-category $\cX \otimes \cE_{\geqslant n}$ is prestable and that
		\begin{equation*} 
			i_{\geqslant n}^\cX \colon \cX \otimes \cE_{\geqslant n} \to \cX \otimes \cE 
		\end{equation*}
		is fully faithful.
	\end{enumerate}
\end{lem}

\begin{proof}
	The equivalence (1)$ \Leftrightarrow $(2) is immediate from the definition of the connective part $ (\cX \otimes \cE)_{\geqslant 0} $ of the standard $ t $-structure.
	The implication (2)$ \Rightarrow $(3) is clear. 

	To see that (3)$ \Rightarrow $(2), without loss of generality, we can suppose $n = 0$.
	In virtue of \cref{cor:standard_t_structure_fundamentals}-(2), we see that $(\cX \otimes \cE)_{\geqslant 0}$ is generated under colimits and extensions by the essential image of $ i_{\geqslant 0}^{\cX} $.
	Thus, to prove that the inclusion $(\cX \otimes \cE)_{\geqslant 0} \subseteq \cX \otimes \cE_{\geqslant 0}$ holds, it suffices to prove that the essential image of $i_{\geqslant 0}^\cX$ is closed under extensions.
	To see this, let $F' \to F \to F''$ be a fiber sequence in $\cX \otimes \cE$ and assume that $F'$ and $F''$ belong to the essential image of $i_{\geqslant 0}^\cX$ (and hence of $i_{\geqslant -1}^\cX$).
	Let $\alpha \colon F' \to F''[1]$ be the map classifying the given extension.
	Since $\cX \otimes \cE$ is stable, we can write
	\begin{equation*} 
		F \simeq \fib\big( F'' \stackrel{\alpha}{\to} F'[1] \big) \comma 
	\end{equation*}
	where $\alpha$ is a morphism in $\cX \otimes \cE$.
	By assumption, we can write
	\begin{equation*} 
		F' \simeq i_{\geqslant 0}^\cX(U') \qquad \text{and} \qquad F'' \simeq i_{\geqslant 0}^\cX(U'') \period 
	\end{equation*}
	Since $i_{\geqslant 0}^\cX$ commutes with colimits, it commutes in particular with suspensions, so that
	\begin{equation*} 
		F'[1] \simeq i_{\geqslant 0}^\cX(U'[1]) \comma 
	\end{equation*}
	where the suspension $U'[1]$ is computed in $\cX \otimes \cE_{\geqslant 0}$.
	The full faithfulness of $i_{\geqslant 0}^\cX$ guarantees that we can write $\alpha \simeq i_{\geqslant 0}^\cX(\beta)$, where $\beta \colon U'' \to U'[1]$ is a morphism in $\cX \otimes \cE_{\geqslant 0}$.
	Set
	\begin{equation*} 
		U \coloneqq \fib(\beta) \in \cX \otimes \cE_{\geqslant 0} \period 
	\end{equation*}
	Since this $\infty$-category is prestable by assumption, we deduce that the pullback diagram
	\begin{equation*} 
		\begin{tikzcd}
			U \arrow{r} \arrow{d} & U'' \arrow{d}{\beta} \\
			0 \arrow{r} & U'[1]
		\end{tikzcd} 
	\end{equation*}
	is also a pushout.
	In particular, it is taken to a pushout by $i_{\geqslant 0}^\cX$ and, since $\cX \otimes \cE$ is stable, we deduce that in fact
	\begin{equation*} 
		i_{\geqslant 0}^\cX( U ) \simeq \fib( i_{\geqslant 0}^\cX(\beta) ) \simeq \fib(\alpha) \simeq F \period 
	\end{equation*}
	Thus, $F$ belongs as well to the essential image of $i_{\geqslant 0}^\cX$, whence the conclusion.
\end{proof}

We now record an easy application of \Cref{lem:tensor_product_connectives_extensions}.
For this, the reader may wish to review the definition of a \textit{projectively generated} presentable $ \infty $-category in \cite[Definition 5.5.8.23]{HTT} or \cite[Recollection 2.4]{arXiv:2108.03545}.

\begin{cor}\label{cor:tensor_product_connectives_for_projectively_generated}
	Let $\cE$ be a presentable stable $\infty$-category equipped with an accessible $t$-structure $ (\cE_{\geqslant 0}, \cE_{\leqslant 0}) $, and let $ \cX $ be a projectively generated presentable $ \infty $-category.
	Then the functor
	\begin{equation*}
		i_{\geqslant 0}^{\cX} \colon \cX \otimes \cE_{\geqslant 0} \to (\cX \otimes \cE)_{\geqslant 0}
	\end{equation*}
	is an equivalence.
\end{cor}

\begin{proof}
	Write $ \cX_{0} \subset \cX $ for the full subcategory spanned by the compact projective objects.
	Then we have identifications
	\begin{equation*}
		\cX \otimes \cE_{\geqslant 0} \simeq \Fun^{\times}(\cX_{0}\op,\cE_{\geqslant 0}) \qquad \text{and} \qquad \cX \otimes \cE \simeq \Fun^{\times}(\cX_{0}\op,\cE) \comma
	\end{equation*}
	where $ \Fun^{\times}(\cX_{0}\op,\cD) $ denotes the full subcategory of $ \Fun(\cX_{0}\op,\cD) $ spanned by the functors that preserve finite products.
	Moreover, since $ i_{\geqslant 0} \colon \cE_{\geqslant 0} \hookrightarrow \cE $ preserves finite products,  under these identifications, the functor $ i_{\geqslant 0}^{\cX} = \id_{\cX} \otimes i_{\geqslant 0} $ is given by postcomposition with $ i_{\geqslant 0} $.
	See \cite[Variant 2.10]{arXiv:2108.03545}.

	In particular, since $ i_{\geqslant 0} $ is fully faithful with essential image closed under extensions, we deduce that $ i_{\geqslant 0}^{\cX} \colon \cX \otimes \cE_{\geqslant 0} \to \cX \otimes \cE $ is fully faithful and the essential image closed under extensions.
	\Cref{lem:tensor_product_connectives_extensions} completes the proof.
\end{proof}

%-------------------------------------------------------------------%
%  An unstable statement                                            %
%-------------------------------------------------------------------%

\subsection{An unstable statement}

Let $\cX$ be an $\infty$-topos.
We write $\mathbf{1}_\cX$ for the final object of $\cX$ and we write
\begin{equation*} 
	\cX_\ast \coloneqq \cX_{\mathbf{1}_\cX /} 
\end{equation*}
for the $\infty$-category of pointed objects of $\cX$.
For an integer $n \geqslant -2$, we consider the $\infty$-category
\begin{equation*} 
	\cX_{\ast}^{\leqslant n} \coloneqq ( \cX_\ast )_{\leqslant n} 
\end{equation*}
of $n$-truncated objects in $\cX_\ast$.
Unraveling the definitions, we see that the inclusion $\cX_\ast^{\leqslant n} \subseteq \cX_\ast$ has a left adjoint that sends a pointed object $(X,x)$ to the pointed object $(\tau_{\leqslant n}(X), x')$, where $x'$ is the composite
\begin{equation*} 
	\begin{tikzcd}[column sep=small]
		\mathbf{1}_\cX \arrow{r}{x} & X \arrow{r} & \tau_{\leqslant n} X \period
	\end{tikzcd} 
\end{equation*}
We still denote this left adjoint by
\begin{equation*} 
	\tau_{\leqslant n} \colon \cX_\ast \to \cX_\ast^{\leqslant n} \period 
\end{equation*}
For $k \geqslant -1$, we define $\cX_\ast^{\geqslant k}$ as the fiber product
\begin{equation*} 
	\begin{tikzcd}
		\cX_\ast^{\geqslant k} \arrow{r} \arrow{d} & \cX_\ast \arrow{d}{\tau_{\leqslant k-1}} \\
		\ast \arrow[r, "\mathbf{1}_\cX"'] & \cX_\ast^{\leqslant k-1} \period
	\end{tikzcd} 
\end{equation*}
The functoriality of the tensor product of $\infty$-categories immediately yields the following commutative diagram:
\begin{equation*} 
	\begin{tikzcd}[column sep=5em]
		\cX \otimes \cS_\ast^{\geqslant k} \arrow{r} \arrow[d, "\alpha_k"'] & \cX \otimes \cS_\ast \arrow{d} \arrow{r}{\id_\cX \otimes \tau_{\leqslant k-1}} & \cX \otimes \cS_\ast^{\leqslant k-1} \arrow{d} \\
		\cX_\ast^{\geqslant k} \arrow{r} & \cX_\ast \arrow[r, "\tau_{\leqslant k-1}^\cX"'] & \cX_\ast^{\leqslant k} \period
	\end{tikzcd} 
\end{equation*}
The central and the right vertical functors are equivalences (see \cite[Examples 4.8.1.21 \& 4.8.1.22]{Lurie_Higher_algebra}), and they would be even if $\cX$ were simply a presentable $\infty$-category.
Since $\cX$ is an $\infty$-topos, we furthermore see:

\begin{prop}\label{prop:coconnectives_topos_tensor_product}
	Let $ \cX $ be an $ \infty $-topos.
	Then for each integer $k > 0$, the comparison functor
	\begin{equation*}
		\alpha_k \colon \cX \otimes \cS_\ast^{\geqslant k} \to \cX_\ast^{\geqslant k}
	\end{equation*}
	is an equivalence.
\end{prop}

\begin{proof}
	Recall from \cite[Notation 5.2.6.11]{Lurie_Higher_algebra} the iterated bar-cobar adjunction
	\begin{equation*} 
		\Bar^{(k)}_\cX \colon \Mon_{\EE_k}(\cX) \leftrightarrows \cX_\ast \colon \CoBar_\cX^{(k)} \period 
	\end{equation*}
	(For $ X \in \cX_{\ast} $, the underlying object of $ \CoBar_\cX^{(k)}(X) $ is just the $ k $-fold based loop object $ \Omega^k X $.)
	By \cite[Theorem 5.2.6.15]{Lurie_Higher_algebra}, this adjunction restricts to an equivalence 
	\begin{equation*} 
		\Bar^{(k)}_\cX \colon \Mon_{\EE_k}^{\gp}(\cX) \leftrightarrows \cX_\ast^{\geqslant k} \colon \CoBar_\cX^{(k)} \period 
	\end{equation*}
	Notice that
	\begin{align*}
		\cX \otimes \Mon_{\EE_k}^{\gp}(\cS) &\simeq \FunR( \cX\op, \Mon_{\EE_k}^{\gp}(\cS) ) \\
		&\simeq \Mon_{\EE_k}^{\gp}(\FunR(\cX\op, \cS)) \\
		&\simeq \Mon_{\EE_k}^{\gp}(\cX) \period
	\end{align*}
	Hence it suffices to argue that the diagram
	\begin{equation*} 
		\begin{tikzcd}[sep=3em]
			\cX \otimes \Mon_{\EE_k}^{\gp}( \cS ) \arrow[r, "\sim"{yshift=-0.25ex}] \arrow[d, "\id_\cX \otimes \Bar^{(k)}"'] & \Mon_{\EE_k}^{\gp}(\cX) \arrow{d}{\Bar_\cX^{(k)}} \\
			\cX \otimes \cS^{\geqslant k}_\ast \arrow[r, "\alpha_k"'] & \cX_\ast^{\geqslant k}
		\end{tikzcd} 
	\end{equation*}
	commutes.
	(Here, $\Bar^{(k)}$ denotes the iterated bar construction for $ \cS$.)
	Since all functors commute with colimits, it suffices to check that the diagram commutes after composition with the universal functor
	\begin{equation*} 
		\cX \times \Mon_{\EE_k}^{\gp}(\cS) \to \cX \otimes \Mon_{\EE_k}^{\gp}(\cS)  
	\end{equation*}
	that preserves colimits separately in each variable.
	For this, it is enough to observe that given $X \in \cX$, the functor
	\begin{equation*} 
		X \otimes (-) \colon \cS \to \cX 
	\end{equation*}
	commutes with colimits and therefore \cite[Example 5.2.3.11]{Lurie_Higher_algebra} supplies a canonical identification
	\begin{equation*} 
		X \otimes \Bar^{(k)}(-) \simeq \Bar^{(k)}_\cX(X \otimes -) \comma 
	\end{equation*}
	which is functorial in $X$.
	The conclusion follows.
\end{proof}

\begin{cor}\label{cor:prestability_sheaves_of_spectra}
	Let $\cX$ be an $\infty$-topos.
	Then the natural functor
	\begin{equation*} 
		\cX \otimes \Sp_{\geqslant 0} \to \cX \otimes \Sp 
	\end{equation*}
	is fully faithful.
\end{cor}

\begin{proof}
	Recall from \cite[Remark 5.2.6.26]{Lurie_Higher_algebra} that one has
	\begin{equation*} 
		\begin{tikzcd}[column sep=small]
			\Sp_{\geqslant 0} \simeq \lim \big( \cdots \arrow{r}{\Omega} & \cS_\ast^{\geqslant n+1} \arrow{r}{\Omega} & \cS_{\ast}^{\geqslant n} \arrow{r}{\Omega} & \cdots \arrow{r}{\Omega} & \cS_\ast^{\geqslant 1} \big) \period
		\end{tikzcd} 
	\end{equation*}
	Similarly,
	\begin{equation*} 
		\begin{tikzcd}[column sep=small]
			\Sp \simeq \lim \big( \cdots \arrow{r}{\Omega} & \cS_\ast \arrow{r}{\Omega} & \cS_{\ast} \arrow{r}{\Omega} & \cdots \arrow{r}{\Omega} & \cS_\ast \big) \period
		\end{tikzcd} 
	\end{equation*}
	Moreover, the inclusion $ \Sp_{\geqslant 0} \hookrightarrow \Sp$ is induced by the fully faithful inclusions $\cS_\ast^{\geqslant n} \hookrightarrow \cS_\ast$, which assemble into a natural transformation of the above limit diagrams.
	Notice that both limits are taken in $\PrR$ and therefore they are preserved by the functor $\cX \otimes (-) $ (see \cite[Remark 4.8.1.24]{Lurie_Higher_algebra}).
	Thus, the claim follows at once from \cref{prop:coconnectives_topos_tensor_product}.
\end{proof}

\begin{rem}
	Let $(\cC, \tau)$ be an $\infty$-site and assume that $\cX \simeq \Sh(\cC,\tau)$.
	The functoriality of the tensor product in $\PrL$ immediately implies that the diagram
	\begin{equation*} 
		\begin{tikzcd}
			\PSh(\cC) \otimes \Sp_{\geqslant 0} \arrow{r} \arrow{d} & \PSh(\cC) \otimes \Sp \arrow{d} \\
			\Sh(\cC,\tau) \otimes \Sp_{\geqslant 0} \arrow{r} & \Sh(\cC,\tau) \otimes \Sp
		\end{tikzcd} 
	\end{equation*}
	commutes, where the vertical arrows are the sheafification functors.
	Moreover, the formula for the sheafification provided in the proof of \cite[Proposition 6.2.2.7]{HTT} (which holds with coefficients in any presentable $\infty$-category) shows that this square is horizontally right adjointable.
	In particular, one can deduce the full faithfulness provided by \cref{cor:prestability_sheaves_of_spectra} for $ \cX = \Sh(\cC,\tau)$ directly from the one for $ \cX = \PSh(\cC)$, which is straightforward since $\PSh(\cC) \otimes (-) \simeq \Fun(\cC\op, -)$ preserves fully faithful left adjoints.
\end{rem}

%-------------------------------------------------------------------%
%  Proof of the main theorem                                        %
%-------------------------------------------------------------------%

\subsection{Proof of the main theorem}

\begin{recollection}\label{recollection:reduction_to_spectra}
	Let $\cY$ be a presentable $\infty$-category and let $\cD$ be a presentable prestable $\infty$-category.
	Then combining \cite[Example C.1.5.6 and Theorem C.4.1.1]{Lurie_SAG} we deduce that
	\begin{align*} 
		\cY \otimes \cD &\simeq \cY \otimes (\Sp_{\geqslant 0} \otimes \cD) \\
		&\simeq (\cY \otimes \Sp_{\geqslant 0}) \otimes \cD \period 
	\end{align*}
	Similarly, if $\cE$ is a presentable stable $\infty$-category, \cite[Example 4.8.1.23]{Lurie_Higher_algebra} shows that
	\begin{align*} 
		\cY \otimes \cE &\simeq \cY \otimes (\Sp \otimes \cE) \\
		&\simeq (\cY \otimes \Sp) \otimes \cE \\
		&\simeq \Stab(\cY) \otimes \cE \period
	\end{align*}
\end{recollection}

\begin{cor}\label{cor:prestability}
	Let $\cX$ be an $\infty$-topos and let $\cD$ be a Grothendieck prestable $\infty$-category.
	Then $\cX \otimes \cD$ is again a Grothendieck prestable $\infty$-category.
\end{cor}

\begin{proof}
	Combining \cref{recollection:reduction_to_spectra} and \cite[Theorem C.4.2.1]{Lurie_SAG}, it is enough to deal with the case where $\cD = \Sp_{\geqslant 0}$, and this case immediately follows from \cref{cor:prestability_sheaves_of_spectra}, \cref{eg:sheaves_of_spectra}, and \cite[Proposition C.1.2.9]{Lurie_SAG}.
\end{proof}

\begin{prop}\label{cor:standard_t_structure_topos}
	Let $\cX$ be an $\infty$-topos.
	Let $\cE$ be a presentable stable $\infty$-category equipped with an accessible $t$-structure $ \tau = (\cE_{\geqslant 0}, \cE_{\leqslant 0})$ which is compatible with filtered colimits and right complete.
	Then for every integer $n \in \ZZ$, the natural functor
	\begin{equation*} 
		\cX \otimes \cE_{\geqslant n} \to (\cX \otimes \cE)_{\geqslant n} 
	\end{equation*}
	is an equivalence.
	In particular, the standard $t$-structure on $\cX \otimes \cE$ is compatible with filtered colimits and right complete.
\end{prop}

\begin{proof}
	It is enough to treat the case $n = 0$.
	We know from \cref{cor:prestability} that $\cX \otimes \cE_{\geqslant 0}$ is a Grothendieck prestable $\infty$-category.
	In particular, \cite[Remark C.1.1.6 \& Proposition C.1.2.9]{Lurie_SAG} imply that the natural functor
	\begin{equation*} 
		\cX \otimes \cE_{\geqslant 0} \to \Stab( \cX \otimes \cE_{\geqslant 0} ) \simeq \cX \otimes \cE_{\geqslant 0} \otimes \Sp \simeq \cX \otimes \Stab(\cE_{\geqslant 0}) 
	\end{equation*}
	is fully faithful.
	On the other hand, since the $t$-structure $\tau$ is right complete, \cite[Remark C.3.1.5]{Lurie_SAG} provides a canonical equivalence $\Stab(\cE_{\geqslant 0}) \simeq \cE$.
	The first claim then follows from \cref{lem:tensor_product_connectives_extensions}.
	Finally, \cite[Proposition C.1.4.1]{Lurie_SAG} guarantees that the unique $t$-structure on $\cX \otimes \cE \simeq \Stab(\cX \otimes \cE_{\geqslant 0})$ whose connective part is given by $\cX \otimes \cE_{\geqslant 0}$ is compatible with filtered colimits.
	
	We are left to prove that the standard $t$-structure is right complete.
	For this, we have to check that the canonical functor
	\begin{equation*} 
		\colim \big( \cdots \to (\cX \otimes \cE)_{\geqslant n} \to (\cX \otimes \cE)_{\geqslant n-1} \to \cdots \ \big) \to \cX \otimes \cE 
	\end{equation*}
	is an equivalence, where the colimit is computed in $\PrL$.
	Using the equivalences
	\begin{equation*}
		(\cX \otimes \cE)_{\geqslant n} \simeq \cX \otimes \cE_{\geqslant n} \comma
	\end{equation*}
	the conclusion follows immediately from the fact that the $t$-structure $\tau$ is right complete and the fact that $\cX \otimes (-) $ commutes with colimits in $\PrL$.
\end{proof}

\begin{rem}
	In particular, \Cref{cor:standard_t_structure_topos} establishes the full faithfulness of the natural functor
	\begin{equation*} 
		\cX \otimes \cE_{\geqslant 0} \to \cX \otimes \cE \period 
	\end{equation*}
	Assume that $\cX = \Sh(\cC,\tau)$ is the $\infty$-topos of sheaves on some $\infty$-site $(\cC,\tau)$.
	Then
	\begin{equation*} 
		\cX \otimes \cE_{\geqslant 0} \simeq \Sh(\cC,\tau;\cE_{\geqslant 0}) \qquad \text{and} \qquad \cX \otimes \cE \simeq \Sh(\cC,\tau;\cE) \period 
	\end{equation*}
	Notice that the natural functor
	\begin{equation*} 
		\Sh(\cC,\tau;\cE_{\geqslant 0}) \to \Sh(\cC,\tau;\cE) 
	\end{equation*}
	induced by the functoriality of the tensor product in $\PrL$ implicitly involves sheafification.
	Indeed, if $F$ is a sheaf with values in $\cE_{\geqslant 0}$, we can view $ F $ as a presheaf with values in $\cE$, but this presheaf is typically not a sheaf (as the constant sheaf on $ \mathrm{S}^1$ with coefficients in a commutative ring $R$ shows).
	Instead, the above comparison functor further sheafifies the resulting presheaf.
	As a result, even for sheaf $ \infty $-topoi, it is not obvious that this functor is fully faithful.
\end{rem}

\begin{cor}\label{cor:geometric_morphism_t_exact}
	Let $f^{*} \colon \cX \to \cY$ be a left exact left adjoint between $\infty$-topoi.
	Let $\cE$ be a presentable stable $\infty$-category equipped with an accessible $t$-structure $\tau = (\cE_{\geqslant 0}, \cE_{\leqslant 0})$ which is compatible with filtered colimits and right complete.
	Then the induced functor
	\begin{equation*} 
		f^{*} \otimes \id_\cE \colon \cX \otimes \cE \to \cY \otimes \cE 
	\end{equation*}
	is $t$-exact.
\end{cor}

\begin{proof}
	We already know from \cref{cor:t_structure_right_exact} that $f^{*} \otimes \id_\cE$ is right $t$-exact.
	To prove left $t$-exactness, we first recall that \cref{cor:standard_t_structure_topos} shows that the $t$-structures on both $\cX \otimes \cE$ and $\cY \otimes \cE$ are right complete.
	Therefore, $\cX \otimes \cE \simeq \Stab(\cX \otimes \cE_{\geqslant 0})$, and similarly for $\cY \otimes \cE$.
	Invoking \cite[Proposition C.3.2.1]{Lurie_SAG}, we see that $f^{*} \otimes \id_\cE$ is left $t$-exact if and only if the induced functor
	\begin{equation*}
		(f^{*} \otimes \id_{\cE_{\geqslant 0}}) \colon \cX \otimes \cE_{\geqslant 0} \to \cY \otimes \cE_{\geqslant 0} 
	\end{equation*}
	is left exact.	
	Combining \cref{recollection:reduction_to_spectra}, \cref{cor:prestability} and \cite[Proposition C.4.4.1]{Lurie_SAG}, we reduce ourselves to the case where $\cE = \Sp$.
	In this case, the conclusion follows from the second half of \cref{cor:t_structure_right_exact}.
\end{proof}

\bibliographystyle{amsalpha}
\bibliography{dahema}

\def\cprime{$'$}
\providecommand{\bysame}{\leavevmode\hbox to3em{\hrulefill}\thinspace}
\providecommand{\MR}{\relax\ifhmode\unskip\space\fi MR }
\providecommand{\MRhref}[2]{%
  \href{http://www.ams.org/mathscinet-getitem?mr=#1}{#2}
}
\providecommand{\href}[2]{#2}
\begin{thebibliography}{GHN17}

\bibitem[GHN17]{Gepner_Lax_colimits}
David Gepner, Rune Haugseng, and Thomas Nikolaus, \emph{Lax colimits and free
  fibrations in {$\infty$}-categories}, Doc. Math. \textbf{22} (2017),
  1225--1266,
  \href{https://arxiv.org/abs/1501.02161}{\nolinkurl{arXiv:1501.02161}}.
  \MR{3690268}

\bibitem[Hai22]{arXiv:2108.03545}
Peter~J. Haine, \emph{From nonabelian basechange to basechange with
  coefficients},
  \href{https://arxiv.org/abs/2108.03545}{\nolinkurl{arXiv:2108.03545}},
  September 2022.

\bibitem[HPT24]{Beyond_conicality}
Peter~J. Haine, Mauro Porta, and Jean-Baptiste Teyssier, \emph{Exodromy beyond
  conicality},
  \href{https://arxiv.org/abs/2401.12825}{\nolinkurl{arXiv:2401.12825}},
  January 2024.

\bibitem[Lur09]{HTT}
Jacob Lurie, \emph{Higher topos theory}, Annals of Mathematics Studies, vol.
  170, Princeton University Press, Princeton, NJ, 2009. \MR{2522659
  (2010j:18001)}

\bibitem[Lur17]{Lurie_Higher_algebra}
\bysame, \emph{Higher algebra},
  \href{http://www.math.ias.edu/~lurie/papers/HA.pdf}{\nolinkurl{math.ias.edu/~lurie/papers/HA.pdf}},
  September 2017.

\bibitem[Lur18]{Lurie_SAG}
\bysame, \emph{Spectral algebraic geometry},
  \href{http://www.math.ias.edu/~lurie/papers/SAG-rootfile.pdf}{\nolinkurl{math.ias.edu/~lurie/papers/SAG-rootfile.pdf}},
  February 2018.

\end{thebibliography}

\end{document}